\documentclass[a4paper,english]{article}
\usepackage[a4paper]{geometry}
\usepackage{amsmath,amsthm,amssymb,mathtools}
\usepackage{enumerate,csquotes,verbatim}
\usepackage{setspace}
\usepackage[numbers]{natbib}
\usepackage{hyperref}
\usepackage{color,xcolor}
\usepackage{tikz}
\usetikzlibrary{arrows}
\usepackage{multirow}
\usepackage{graphicx}
\usepackage{tabularx}
\usepackage{float} 		
\restylefloat{table}    

\newtheorem{thm}{Theorem}
\newtheorem{lem}[thm]{Lemma}

\newtheorem{con}[thm]{Conjecture}  

\theoremstyle{definition}

\newtheorem{prob}[thm]{Problem}

\DeclareMathOperator{\diam}{diam} 
\DeclareMathOperator{\dbar}{\overrightarrow{d}}
\DeclareMathOperator{\Gbar}{\overrightarrow{G}} 

\title{A Size Condition for Small Diameter Orientable Graphs}

\author{Sopon Boriboon\thanks{\,Thammasat Secondary School, Faculty of Learning Sciences and Education, Thammasat University, Pathum Thani 12121, Thailand; \texttt{sopon.bo@lsed.tu.ac.th}.}
  \and Teeradej Kittipassorn\thanks{\,Department of Mathematics and Computer Science, Faculty of Science, Chulalongkorn University, Bangkok 10330, Thailand; \texttt{teeradej.k@chula.ac.th}.}}

\begin{document}
	
	\maketitle
	
\begin{abstract}
	
In 2002, Koh and Tay conjectured that every bridgeless graph of order $n\geq 5$ and size at least ${n\choose 2}-n+5$ has an orientation of diameter two. Later, Cochran,  Czabarka, Dankelmann  and Sz\'{e}kely proved this conjecture and asked what is the minimum number of edges required in a bridgeless graph of order $n$ to guarantee the existence of an orientation of diameter at most $d$?  We conjecture that the answer is ${n-d \choose 2}+n+2$. We prove this conjecture for the case $d=n-2$ and prove the lower bound of this conjecture for the case $5\leq d\leq n-2$.
	
\end{abstract}
	
\section{Introduction}\label{section-introduction}
\label{thm:robbin}

An edge $e$ of a connected graph $G$ is  a {\it bridge} if $G-e$ is disconnected.  A graph $G$ is \emph{bridgeless} if it is connected and contains no bridge. An \emph{orientation} $O_G$ of a graph $G$ is a digraph $\Gbar$ obtained from $G$ by assigning a direction $\overrightarrow{uv}$ or $\overrightarrow{vu}$ for each edge $uv$ of $G$.  For vertices $u$ and $v$ in a digraph $\Gbar$, the \emph{distance} from $u$ to $v$ in $\Gbar$, denoted by $\dbar_{\Gbar}(u,v)$, is the minimum length of a $u,v$-path on $\Gbar$ and $\dbar_{\Gbar}(u,v)=\infty$ if $\Gbar$ has no such path. The \emph{diameter} of a digraph $\Gbar$, denoted by $\diam(\Gbar)$, is $\max_{u,v\in V(\Gbar)}\dbar_{\Gbar}(u,v)$.

Applications of optimal orientations that minimize the diameter can be found in various areas, such as the design of efficient one-way street traffic systems. The two-way street traffic system can be effectively represented by a graph where vertices denote the street intersections, and edges denote the direct routes between intersections without traversing any intermediate intersections.   Robbin~\cite{Robbins1939Questions} showed that the two-way street traffic system can be converted to a one-way traffic system if and only if it is possible to travel in the two-way street traffic system from any intersection to another even when any street is removed by providing a necessary and sufficient condition for the existence of an orientation of finite diameter of the graph.

\begin{thm}\label{thm:robbin}\cite{Robbins1939Questions}
	A connected graph has an orientation of finite diameter if and only if it is bridgeless. 
\end{thm}

There are many researches regarding sufficient conditions for the existence of an orientation of diameter of graphs, e.g., diameter conditions~\cite{Egawa2007Orientation,kwok2010oriented}, minimum degree conditions~\cite{Bau2015Diameter,Chen2021Diameter,Czabarka2019Degree,Surmacs2017Improved}, maximum degree conditions~\cite{Chen2023Oriented,Dankelmann2018Oriented}, degree conditions~\cite{Czabarka2019Degree}, girth conditions~\cite{Chen2023Oriented,Cochran2024Large}, block conditions~\cite{Dankelmann2025Diameter}, and size conditions~\cite{Cochran2021Size}.

For given values $n$ and $d$, let $m(n,d)$ be the minimum number $m$ such that every bridgeless graph of order $n$ and size at least $m$ has an orientation of diameter at most $d$. We observe that $m(n,1)$ is clearly infinity and for the cases $d\geq n-1$, $m(n,d)$ is negative infinity by Theorem~\ref{thm:robbin} since the diameter of bridgeless graphs is always at most $n-1$. Therefore, we consider $m(n,d)$ where $2\leq d\leq n-2$. In 2002, Koh and Tay~\cite{Koh2002Optimal} observed that the graph of order $n$ and size ${n\choose 2}-n+4$ obtained from $K_{n-1}$ by adding a new vertex $v$ joining to three vertices of $K_{n-1}$ has no orientation of diameter two, therefore $m(n,2)\geq{n\choose 2}-n+5$. They conjectured that every bridgeless graph of order $n\geq5$ and size at least ${n\choose 2}-n+5$ has an orientation of diameter two. In 2021, Cochran,  Czabarka, Dankelmann and Sz\'{e}kely~\cite{Cochran2021Size} proved this conjecture, implying that $m(n,2)={n\choose 2}-n+5$. Moreover, they asked for the exact value of $m(n,d)$ for more general value of $d$.

In this paper, we show that the minimum number of edges required in any bridgeless graph of order $n$ to guarantee the existence of an orientation of diameter at most $n-2$ is $n+3$. 

\begin{thm}\label{thm:n-2}
	For  $n\geq 5$,  we have $m(n,n-2) \leq n+3$.	
\end{thm}

On the other hand, we demonstrate a graph of order $n$ and size ${n-d\choose 2}+n+1$ without an orientation of diameter at most $d$ for any $5\leq d \leq n-2$.

\begin{thm}\label{thm:sharp}
	
	\begin{enumerate}[(i)]
		\item For  $5\leq d\leq n-2$,  we have $m(n,d) \geq {n-d\choose 2}+n+2$.	
		\item For  $n\geq 5$,  we have $m(n,n-2) \geq n+3$.	
	\end{enumerate}
\end{thm}
 
By Theorems~\ref{thm:n-2} and~\ref{thm:sharp}, for $n\geq 5$, we have  

\begin{center}
	${\displaystyle m(n,n-2) = n+3 = {n-(n-2)\choose 2}+n+2}$.
\end{center}

Recall that Cochran,  Czabarka, Dankelmann and Sz\'{e}kely~\cite{Cochran2021Size} showed that

\begin{center}
	${\displaystyle m(n,2) = {n\choose 2}-n+5 =  {n-2 \choose 2}+n+2}$.
\end{center}

These motivate us to conjecture the following form for $m(n,d)$.	

\begin{con}\label{con:m(n,d)}
	For  $2\leq d\leq n-2$,  we have $m(n,d)={n-d\choose 2}+n+2$.	
\end{con}	

The rest of this paper is organized as follows. Section~\ref{section-proofs} is dovoted to proving Theorems~\ref{thm:n-2} and~\ref{thm:sharp}. In Section~\ref{section-concludingremark}, we give some concluding remarks.

\section{Proofs}\label{section-proofs} 

Before we prove Theorems~\ref{thm:n-2} and~\ref{thm:sharp}, we give some definitions necessary for the proofs. Let $P_n$ denote the path on $n$ vertices. An orientation $O_G$ of a graph $G$ is  \emph{strong}  if for each $u,v\in V(G)$, there is a $u,v$-path in $\Gbar$, and $\Gbar$ is said to be \emph{strongly connected}. For vertices $u,v$ in a digraph $\overrightarrow{G}$, a vertex $u$ is an \textit{in-neighbor} of $v$ if $\overrightarrow{uv}\in E(\overrightarrow{G})$ and an \textit{out-neighbor}  of $v$ if $\overrightarrow{vu}\in E(\overrightarrow{G})$. For vertices $x,y$ in a directed path $\overrightarrow{P}$, we denote by $\overrightarrow{P}[x,y]$ the $x,y$-subpath of $\overrightarrow{P}$ (if exists) and let $\overrightarrow{P}(x,y]=\overrightarrow{P}[x,y]-x$ and $\overrightarrow{P}[x,y)=\overrightarrow{P}[x,y]-y$. For vertices $x,y$ in a directed cycle $\overrightarrow{C}$, we denote by $\overrightarrow{C}[x,y]$ the $x,y$-subpath of $\overrightarrow{C}$ and let $\overrightarrow{C}(x,y]=\overrightarrow{C}[x,y]-x$ and $\overrightarrow{C}[x,y)=\overrightarrow{C}[x,y]-y$.  For an $x,y$-walk $\overrightarrow{P}=xu_1u_2\ldots y$ and a $y,z$-walk $\overrightarrow{Q}=yv_1v_2\ldots z$ sharing a common endpoint $y$, we let $\overrightarrow{P}+\overrightarrow{Q}$ denote the directed walk  $xu_1u_2\ldots yv_1v_2\ldots z$. For an $x,y$-path in  $\overrightarrow{P}[x,y]$, the \textit{reverse path} of  an $x,y$-path in  $\overrightarrow{P}[x,y]$, denoted by  $\overleftarrow{P}[x,y]$, is the $y,x$-path obtained from redirecting all edges in  $\overrightarrow{P}[x,y]$. We write $\overrightarrow{G}[x,y]$ to represent some $x,y$-path in  $\overrightarrow{G}$ (if exists). The union $G_1 \cup G_2$ of graphs $G_1 = (V_1, E_1)$ and $G_2 = (V_2, E_2)$ is the graph with vertex set $V_1 \cup V_2$ and edge set $E_1 \cup E_2$. We denote the set $\{1,2,3,\ldots,k\}$ by $[k]$.

We start by providing examples of bridgeless graphs without an orientation of diameter at most $d$. For a given value $5\leq d\leq n-2$,  we define the graph $G_{n,d}$ of order $n$ and size ${n-d\choose 2}+n+1$ as follows. Join both endpoints of a path $u_1u_2u_3\ldots  u_d$ to all vertices of $K_{n-d-1}$ and add a vertex $u'_3$ joining to $u_2$, $u_3$ and $u_4$. Clearly, $G_{n,d}$ is bridgeless. We observe that the size of $G_{n,d}$ is 

\begin{center}
	${\displaystyle e(G_{n,d})={n-d-1\choose 2}+(d-1)+2(n-d-1)+3={n-d\choose 2}+n+1}$.
\end{center}

We would like to extend $G_{n,n-2}$ for $5\leq n\leq6$. For a given value $n\geq 5$, we define the graph $H_{n,n-2}$ of order $n$ and size $n+2$ as follows. Join a vertex $u'_3$ to $u_2$, $u_3$ and $u_4$ of a cycle $u_1u_2u_3\ldots  u_{n-1}$. Clearly, $H_{n,n-2}$ is bridgeless. We observe that the size of $H_{n,n-2}$ is   $e(H_{n,n-2})=n+2$. Note that $G_{n,n-2}=H_{n,n-2}$ for $n\geq 7$.

We will now prove Theorem~\ref{thm:sharp} by showing that $G_{n,d}$ has no orientation of diameter at most $d$ and  $H_{n,n-2}$ has no orientation of diameter at most $n-2$.

\begin{proof}[Proof of Theorem~\ref{thm:sharp}] $(i)$ We suppose to the contrary that $G_{n,d}$ has an orientation of diameter at most $d$. We observe that each vertex of degree two must have exactly one in-neighbor and one out-neighbor. Without loss of generality, $u_iu_{i+1}$ has the orientation $\overrightarrow{u_iu_{i+1}}$ for all $4\leq i \leq d-1$. Since $d\geq 5$, there is a $u_5,u_4$-path of length at most $d$ and this path must pass through 
	\begin{center}
		$u_5$,$u_6$,$\ldots$,$u_d,$$K_{n-d-1}$,$u_1$,$u_2$,$\{u_3,u'_3\}$,$u_4$
	\end{center}
in that order. Therefore, it must use exactly one of $u_3$ and $u'_3$, say $u_3$. Then $u_1u_2$, $u_2u_3$  and $u_3u_4$ have the orientations $\overrightarrow{u_1u_2}$, $\overrightarrow{u_2u_3}$ and $\overrightarrow{u_3u_4}$, respectively.
	
\textit{Case 1.}  $u_3u'_3$ has the orientation $\overrightarrow{u_3u'_3}$.

There is a $u'_3,u_1$-path of length at most $d$ and this path cannot pass through $u_3$ otherwise it is too long. Therefore, it must use the edge $\overrightarrow{u'_3u_4}$. There is a $u_4,u'_3$-path of length at most $d$ and this path cannot pass through $u_3$ otherwise it is too long. Therefore, it must use  the edge $\overrightarrow{u_2u'_3}$. Hence, any $u'_3,u_3$-path must pass through  
	\begin{center}
		$u'_3,u_4,u_5,\ldots,u_d,K_{n-d-1},u_1,u_2,u_3$
	\end{center} 
in that order, i.e., $\dbar_{\overrightarrow{G}_{n,d}}(u'_3,u_3)\geq d+1$ which is a contradiction.     	 

\textit{Case 2.} $u'_3u_3$ has the orientation $\overrightarrow{u'_3u_3}$

There is a $u_3,u'_3$-path of length at most $d$ and this path cannot pass through $\overrightarrow{u_4u_5}$ otherwise it is too long. Therefore, it must use the edge $\overrightarrow{u_4u'_3}$. There is a $u'_3,u_2$-path of length at most $d$ and this path cannot pass through $\overrightarrow{u'_3u_3}$ otherwise it is too long. Therefore, it must use the edge $\overrightarrow{u'_3u_2}$. Hence, any $u_5,u'_3$-path must pass through  
	\begin{center}
		$u_5$,$u_6$,$\ldots$,$u_d$,$K_{n-d-1}$,$u_1$,$u_2$,$u_3$,$u_4$,$u'_3$
	\end{center}
in that order, i.e., $\dbar_{\overrightarrow{G}_{n,d}}(u_5,u'_3)\geq d+1$ which is a contradiction.    
		
Therefore, $G_{n,d}$ has no orientation of diameter at most $d$. Hence, $m(n,d) \geq {n-d\choose 2}+n+2$.	

$(ii)$ The proof is the same as $(i)$ execpt we omit $K_{n-d-1}$ when $3\leq d\leq 4$ and identify $u_5=u_1$ when $d=3$.

Therefore, $H_{n,n-2}$ has no orientation of diameter at most $d$. Hence, $m(n,n-2)\geq n+3$.	\end{proof}
	
We have shown that $m(n,d) \geq {n-d\choose 2}+n+2$ for $5\leq d \leq n-2$. Next, we will prove Theorem~\ref{thm:n-2} which implies that this lower bound is the exact value of $m(n,d)$ for the case $d=n-2$. The proof is divided into two lemmas, namely, Lemmas~\ref{lem:Hamiltonian} and~\ref{lem:non-Hamiltonian},  by the Hamitoniancity of the graph. We first consider Hamiltonian graphs.

\begin{lem}\label{lem:Hamiltonian}
	Let $G$ be a Hamiltonian graph of order $n\geq5$ and size at least $n+3$. Then $G$ has an orientation of diameter at most $n-2$.
\end{lem}

\begin{proof} Let $C$ be a Hamiltonian cycle contained in a bridgeless graph $G$ of order $n$ with  
	
	\begin{center}
		$V(C)=\{v_1,v_2,\dots,v_n\}$ \,\,and\,\,  $E(C)=\{e_1,e_2,\ldots,e_n\}$, 
	\end{center}
	
\noindent where $e_k=v_kv_{k+1}$ for $k\in[n]$ and the indices are taken modulo $n$.  Let $\overrightarrow{C}$ be the directed cycle obtained from $C$ by  $\overrightarrow{e_k}=\overrightarrow{v_kv_{k+1}}$ for $k\in[n]$. We note that any two vertices $v_k,v_\ell$ in $G$,  $\dbar_{\overrightarrow{C}}(v_k,v_\ell)\leq n-2$ if $\ell\neq k-1 $. By $e(G)\geq n+3$ and $e(C)= n$, there exist edges $f_{1}=v_{i_1}v_{j_1}, f_{2}=v_{i_2}v_{j_2}$ and $f_{3}=v_{i_3}v_{j_3}$ in $E(G)\setminus E(C)$.  To show that $G$ has an orientation of diameter at most $n-2$, it suffices to assign the directions to $f_{1}$, $f_{2}$ and $f_{3}$ in order to create a $v_k,v_{k-1}$-path omitting some vertex, for all $k\in[n]$. We divide the proof into two cases.
	
\textit{Case 1.}  $C$ has a pair of parallel chords. (They can share an endpoint.)

\noindent  Without loss of generality, we suppose that $1\leq i_1\leq i_2<j_2<j_1\leq n$.	We orient the edges $f_1$ and $f_2$ as $\overrightarrow{f_1}=\overrightarrow{v_{j_1}v_{i_1}}$ and $\overrightarrow{f_2}=\overrightarrow{v_{i_2}v_{j_2}}$, respectively. Let

 \begin{center}
 	$\overrightarrow{C_1}=\overrightarrow{C}[v_{i_1},v_{j_1}]+\overrightarrow{f_1}$ and $\overrightarrow{C_2}=\overrightarrow{C}[v_{j_2},v_{i_2}]+\overrightarrow{f_2}$.
 \end{center}

\noindent We note that $V(\overrightarrow{C_1})\leq n-2$ and  $V(\overrightarrow{C_2})\leq n-2$. Let $k\in[n]$. We have  $v_k,v_{k-1}\in V(\overrightarrow{C_\ell})$ for some $\ell\in\{1,2\}$.
Therefore, 

\begin{center}
	$\dbar_{\overrightarrow{G}}(v_k,v_{k-1})\leq\dbar_{\overrightarrow{C_\ell}}(v_k,v_{k-1})\leq n-2$. 
\end{center}
Each of the remaining edges in $G$ can be oriented with arbitrary direction. Hence, $G$ has an orientation of diameter at most $n-2$. 

\textit{Case 2.} There is no a pair of parallel edges. 

\noindent Without loss of generality, we suppose that $1\leq i_1<i_2<i_3<j_1<j_2<j_3\leq n$. 	We orient the edges $f_1$, $f_2$ and $f_3$ as  $\overrightarrow{f_1}=\overrightarrow{v_{j_1}v_{i_1}}$,  $\overrightarrow{f_2}=\overrightarrow{v_{i_2}v_{j_2}}$ and $\overrightarrow{f_3}=\overrightarrow{v_{j_3}v_{i_3}}$, respectively. 
Let 

\begin{center}
	$\overrightarrow{C_1}=\overrightarrow{C}[v_{i_1},v_{j_1}]+\overrightarrow{f_1}$,  $\overrightarrow{C_2}=\overrightarrow{C}[v_{j_2},v_{i_2}]+\overrightarrow{f_2}$ and $\overrightarrow{C_3}=\overrightarrow{C}[v_{i_3},v_{j_3}]+\overrightarrow{f_3}$.
\end{center}

\noindent  We note that $V(\overrightarrow{C_\ell})\leq n-2$ for $\ell\in\{1,2,3\}$. Let $k\in[n]$, we have  $v_k,v_{k-1}\in V(\overrightarrow{C_\ell})$ for some $\ell\in\{1,2,3\}$.
Therefore, 

\begin{center}
	$\dbar_{\overrightarrow{G}}(v_k,v_{k-1})\leq\dbar_{\overrightarrow{C_\ell}}(v_k,v_{k-1})\leq n-2$. 
\end{center}
Each of the remaining edges in $G$ can be oriented with arbitrary direction. Hence, $G$ has an orientation of diameter at most $n-2$. 
\end{proof}

Next, we will deal with non-Hamiltonian graphs.

\begin{lem}\label{lem:non-Hamiltonian}
	Let $G$ be a non-Hamiltonian bridgeless graph of order $n\geq5$ and size at least $n+3$. Then $G$ has an orientation of diameter at most $n-2$.
\end{lem}

\begin{proof} Let $G$ be a non-Hamiltonian bridgeless graph of order $n\geq5$ with $e(G)\geq n+3$. By Theorem~\ref{thm:robbin}, $G$ has a strong orientation $O_G$. If $\diam(\Gbar)\leq n-2$, we are done. Otherwise,  there exists an $x_1,x_2$-path $\overrightarrow{P}$ in $\Gbar$ of length $n-1$.  We write
	
\begin{center}
	$\overrightarrow{P}=w_1w_2w_3\ldots w_n$,
\end{center}	
where $w_1=x_1$ and $w_n=x_2$. Since $\Gbar$ is strongly connected, there exist $\overrightarrow{e_1}=\overrightarrow{y_1x_1}$ and $\overrightarrow{e_2}=\overrightarrow{x_2y_2}$ in $E(\Gbar)\setminus E(\overrightarrow{P})$. Without loss of generality, choose $y_1$ and $y_2$ such that $\dbar_{\overrightarrow{P}}(x_1,y_1)$ and $\dbar_{\overrightarrow{P}}(y_2,x_2)$ are maximal. By the fact that $G$ is non-Hamiltonian, we have $y_1\neq x_2$ and $y_2\neq x_1$.  We divide the proof into two cases.
	
\textit{Case 1.}  $\dbar_{\overrightarrow{P}}(x_1,y_2)\leq\dbar_{\overrightarrow{P}}(x_1,y_1)$

Let $u_1$ be the out-neighbor of $y_1$ in $\overrightarrow{P}$ and $u_2$ be the in-neighbor of $y_2$ in $\overrightarrow{P}$ and let

\begin{center}
	$\overrightarrow{Q}=\overrightarrow{P}[u_1,x_2]+\overrightarrow{e_2}+\overrightarrow{P}[y_2,y_1]+\overrightarrow{e_1}+\overrightarrow{P}[x_1,u_2]$.
\end{center}
We write $\overrightarrow{Q}=w'_1w'_2w'_3\ldots w'_n$,
where $w'_1=u_1$ and $w'_n=u_2$.  We will show that $\dbar_{\overrightarrow{G}}(p,q)\leq n-2$ unless $(p,q)\in\{(x_1,x_2),(u_1,u_2)\}$. We set

 \begin{center}
	$\overrightarrow{C_1}=\overrightarrow{P}[x_1,y_1]+\overrightarrow{e_1}$ \,\,and\,\, $\overrightarrow{C_2}=\overrightarrow{P}[y_2,x_2]+\overrightarrow{e_2}$.
\end{center}

\noindent Let $w_i,w_j\in V(\Gbar)$ where $i<j$. From $w_i$ to $w_j$, we have
	
	\begin{center}
	 $\dbar_{\overrightarrow{G}}(w_i,w_j)\leq\dbar_{\overrightarrow{P}}(w_i,w_j)\leq\dbar_{\overrightarrow{P}}(x_1,x_2)= n-1,$ 
	\end{center}
\noindent and  the equality holds only when $(w_i,w_j)=(x_1,x_2)$. From $w_j$ to $w_i$, if $w_i,w_j\in\overrightarrow{C_\ell}$ for some $\ell\in\{1,2\}$, we have

	\begin{center}
		$\dbar_{\overrightarrow{G}}(w_j,w_i)\leq\dbar_{\overrightarrow{C_\ell}}(w_j,w_i)\leq n-2$.
	\end{center}

\noindent Otherwise, since $i<j$, $w_i\in V(\overrightarrow{C_1})\setminus V(\overrightarrow{P}[y_2,y_1])$ and $w_j\in V(\overrightarrow{C_2})\setminus V(\overrightarrow{P}[y_2,y_1])$, i.e., $w_i\in\overrightarrow{P}[x_1,u_2]$ and $w_j\in\overrightarrow{P}[u_1,x_2]$. We have

	\begin{center}
		$\dbar_{\overrightarrow{G}}(w_j,w_i)\leq\dbar_{\overrightarrow{Q}}(w_j,w_i)\leq\dbar_{\overrightarrow{Q}}(u_1,u_2)= n-1,$ 
	\end{center}

\noindent and the equality holds only when ($w_j,w_i)=(u_1,u_2)$. Hence, $\dbar_{\overrightarrow{G}}(p,q)\leq n-2$ unless $(p,q)\in\{(x_1,x_2),(u_1,u_2)\}$.

Since $e(G)\geq n+3$ and $|E(P)\cup E(Q)|= n+1$, there exist $e_3=w_iw_j$ and $e_4=w'_kw'_{\ell}$  in $E(G)\setminus \left(E(P)\cup E(Q)\right)$ where $i<j$ and $k<\ell$. We orient $\overrightarrow{e_3}=\overrightarrow{w_iw_j}$ and  $\overrightarrow{e_4}=\overrightarrow{w'_kw'_\ell}$. Therefore, there exist an $x_1,x_2$-path $\overrightarrow{P}[x_1,w_i]+\overrightarrow{e_3}+\overrightarrow{P}[w_j,x_2]$ omitting some vertex in $\overrightarrow{P}(w_i,w_j)$ and a $u_1,u_2$-path  $\overrightarrow{Q}[u_1,w'_k]+\overrightarrow{e_4}+\overrightarrow{Q}[w'_\ell,u_2]$ omitting some vertex in $\overrightarrow{Q}(w'_k,w'_\ell)$. Hence, $G$ has an orientation of diameter at most $n-2$.

\textit{Case 2.}  $\dbar_{\overrightarrow{P}}(x_1,y_2)>\dbar_{\overrightarrow{P}}(x_1,y_1)$ 

\noindent We note that $\overrightarrow{P}[x_1,y_1]$ and $\overrightarrow{P}[y_2,x_2]$ are disjoint. By the fact that $O_G$ is a strong orientation of $G$, there is an $a_2,a_1$-path $\overrightarrow{R}$ from a vertex in $\overrightarrow{P}[y_2,x_2]$ to a vertex in $\overrightarrow{P}[x_1,y_1]$,
where $V(\overrightarrow{P}[y_2,x_2])\cap V(\overrightarrow{R})=\{a_2\}$ and  $V(\overrightarrow{P}[x_1,y_1])\cap V(\overrightarrow{R})=\{a_1\}$. By the definition of $y_1$ and $y_2$, we have $a_1\neq x_1$ and $a_2\neq x_2$. Let $\overrightarrow{f_1},\overrightarrow{f_2},\dots,\overrightarrow{f_\ell}$ be the edges in $E(\overrightarrow{R})\setminus E(\overrightarrow{P})$ defined in order appearance in $\overrightarrow{R}$.  We write  $\overrightarrow{f_1}=\overrightarrow{a_2b_2}$, $\overrightarrow{f_2}=\overrightarrow{c_2d_2}$  and $\overrightarrow{f_\ell}=\overrightarrow{b_1a_1}$.

\textit{Case 2.1.} 	$\dbar_{\overrightarrow{R}}(a_2,a_1) = 1$. 

\noindent Let $\overrightarrow{S}=\overleftarrow{P}[a_1,a_2]\cup\overleftarrow{R}[a_2,a_1]$ be the union of the reverse paths of $\overrightarrow{P}[a_1,a_2]$ and  $\overrightarrow{R}[a_2,a_1]$. Note that $\overrightarrow{S}$ is a cycle.  We orient a spanning subgraph of $G$ as follows:
	
	\begin{center}
		$\overrightarrow{H}=\overrightarrow{P}[x_1,a_1]\cup \overrightarrow{S}\cup \overrightarrow{P}[a_2,x_2]  \cup\{\overrightarrow{e_1},\overrightarrow{e_2}\}$,
	\end{center}

\noindent and set 

\begin{center}
	$X_1=V(\overrightarrow{P}[x_1,a_1))$, $W=V(\overrightarrow{P}[a_1,a_2])$, $X_2=V(\overrightarrow{P}(a_2,x_2])$.
\end{center}

\noindent We now check that for any two vertices $p$ and $q$ in $G$, there exists a $p,q$-walk omitting some vertex.

For $p\in X_1$ and $q\in W$, $\overrightarrow{P}[p,a_1]+\overrightarrow{S}[a_1,q]$ is a $p,q$-path omitting  $x_2$. 

For $p\in W$ and $q\in X_1$, $\overrightarrow{S}[p,y_1]+\overrightarrow{e_1}+\overrightarrow{P}[x_1,q]$ is a $p,q$-path omitting $x_2$. 

For $p,q\in X_1$, $\overrightarrow{P}[p,a_1]+\overrightarrow{S}[a_1,y_1]+\overrightarrow{e_1}+\overrightarrow{P}[x_1,q]$ is a $p,q$-walk omitting $x_2$.  

For $p\in X_2$ and $q\in W$, $\overrightarrow{P}[p,x_2]+\overrightarrow{e_2}+\overrightarrow{S}[y_2,q]$ is a $p,q$-path omitting $x_1$.

For $p\in W$ and $q\in X_2$, $\overrightarrow{S}[p,a_2]+\overrightarrow{P}[a_2,q]$  is a $p,q$-path omitting $x_1$.  

For $p,q\in X_2$, $\overrightarrow{P}[p,x_2]+\overrightarrow{e_2}+  \overrightarrow{S}[y_2,a_2]+\overrightarrow{P}[a,q]$ is a $p,q$-walk omitting $x_1$.  

For $p,q\in W$, $\overrightarrow{S}[p,q]$ is a $p,q$-path omitting $x_1$ and $x_2$.

For $p\in X_1$ and $q\in X_2$,  $\overrightarrow{P}[p,a_1]+\overleftarrow{R}[a_2,a_1]+\overrightarrow{P}[a_2,q]$  is a $p,q$-path omitting some vertex in  $\overrightarrow{P}(a_1,a_2)$.

For $p\in X_2$ and $q\in X_1$,  let 
\begin{center}
	$\overrightarrow{Q}=\overrightarrow{P}[p,x_2]+\overrightarrow{e_2}+\overrightarrow{S}[y_2,y_1]+\overrightarrow{e_1}+\overrightarrow{P}[x_1,q]$
\end{center}
 be a $p,q$-path.  We write $\overrightarrow{Q}=w'_1w'_2w'_3\ldots w'_k$, where $w'_1=p$ and $w'_k=q$.  If $k<n$,  then $\overrightarrow{Q}$ is a $p,q$-path omitting some vertex. Otherwise, since $e(G)\geq n+3$ and $e(H)= n+2$.  So, there exists an edge $e_3=w'_iw'_j\in E(G)\setminus E(H)$, where $i<j-1$. If we orient the edge $e_3$ as $\overrightarrow{e_3}=\overrightarrow{w'_iw'_j}$, then there is a $p,q$-path  $\overrightarrow{Q}[p,w'_i]+\overrightarrow{e_3}+\overrightarrow{Q}[w'_j,q]$ omitting some vertex in $\overrightarrow{Q}(w'_i,w'_j)$.
 
\textit{Case 2.2.} $\dbar_{\overrightarrow{R}}(a_1,a_2)\geq2$. 

\textit{Case 2.2.1.} $a_2\neq y_2$.

\noindent Let $\overrightarrow{S}=\overleftarrow{P}[a_1,a_2]\cup\overleftarrow{R}[a_2,a_1]$ be the union of the reverse paths of $\overrightarrow{P}[a_1,a_2]$ and  $\overrightarrow{R}[a_2,a_1]$. Note that $\overrightarrow{S}$ is strongly connected since $\overleftarrow{P}[a_1,a_2]+\overleftarrow{R}[a_2,a_1]$ is a spanning closed walk. We orient a spanning subgraph of $G$ as follows: 
	
	\begin{center}
		$\overrightarrow{H}=\overrightarrow{P}[x_1,a_1]\cup \overrightarrow{S}\cup \overrightarrow{P}[a_2,x_2]\cup\{\overrightarrow{e_1},\overrightarrow{e_2}\}$,
	\end{center}
	
\noindent and set 
	
	\begin{center}
		$X_1=V(\overrightarrow{P}[x_1,a_1))$, $W=V(\overrightarrow{P}[a_1,a_2])$, $X_2=V(\overrightarrow{P}(a_2,x_2])$.
	\end{center}
	
\noindent We now check that for any two vertices $p$ and $q$ in $G$, there exists a $p,q$-walk omitting some vertex.

For $p\in X_1$ and $q\in W$,  $\overrightarrow{P}[p,a_1]+\overrightarrow{S}[a_1,q]$ is a $p,q$-path omitting  $x_2$. 

For $p\in W$ and $q\in X_1$,    $\overrightarrow{S}[p,y_1]+\overrightarrow{e_1}+\overrightarrow{P}[x_1,q]$ is a $p,q$-path omitting $x_2$. 

For $p,q\in X_1$,  $\overrightarrow{P}[p,a_1]+\overrightarrow{S}[a_1,y_1]+\overrightarrow{e_1}+\overrightarrow{P}[x_1,q]$ is a $p,q$-walk omitting $x_2$.  

For $p\in X_2$ and $q\in W$,  $\overrightarrow{P}[p,x_2]+\overrightarrow{e_2}+\overrightarrow{S}[y_2,q]$ is a $p,q$-path omitting $x_1$.

For $p\in W$ and $q\in X_2$,   $\overrightarrow{S}[p,a_2]+\overrightarrow{P}[a_2,q]$  is a $p,q$-path omitting $x_1$.  

For $p,q\in X_2$,   $\overrightarrow{P}[p,x_2]+\overrightarrow{e_2}+  \overrightarrow{S}[y_2,a_2]+\overrightarrow{P}[a_2,q]$ is a $p,q$-walk omitting $x_1$. 

For $p,q\in W$,  $\overrightarrow{S}[p,q]$ is a $p,q$-path omitting $x_1$ and $x_2$.

For $p\in X_1$ and $q\in X_2$,   $\overrightarrow{P}[p,a_1]+\overleftarrow{R}[a_2,a_1]+\overrightarrow{P}[a_2,q]$  is a $p,q$-path omitting $y_2$ since $a_2\neq y_2$.

For $p\in X_2$ and $q\in X_1$,   $\overrightarrow{P}[p,x_2]+\overrightarrow{e_2}+\overleftarrow{P}[y_1,y_2]+\overrightarrow{e_1}+\overrightarrow{P}[x_1,q]$  is a $p,q$-path omitting $a_2$ since $a_2\neq y_2$.

\textit{Case 2.2.2.}  $b_2\neq c_2$. 

\noindent Let $\overrightarrow{S}=\overleftarrow{P}[b_2,a_2]\cup\overleftarrow{R}[a_2,b_2]$ be the union of the reverse paths of $\overrightarrow{P}[b_2,a_2]$ and  $\overrightarrow{R}[a_2,b_2]$. Note that $\overrightarrow{S}$ is a cycle. We orient a spanning subgraph of $G$ as follows:

\begin{center}
	$\overrightarrow{H}=\overrightarrow{P}[x_1,b_2]\cup\overrightarrow{R}[c_2,a_1]\cup \overrightarrow{S}\cup \overrightarrow{P}[a_2,x_2]\cup\{\overrightarrow{e_1},\overrightarrow{e_2}\}$,
\end{center}

\noindent and set 

\begin{center}
	$X_1=V(\overrightarrow{P}[x_1,b_2))$, $W=V(\overrightarrow{P}[b_2,a_2])$, $X_2=V(\overrightarrow{P}(a_2,x_2])$.
\end{center}

\noindent We now check that for any two vertices $p$ and $q$ in $G$, there exists a $p,q$-walk omitting some vertex.

For $p\in X_1$ and $q\in W$, $\overrightarrow{P}[p,b_2]+\overrightarrow{S}[b_2,q]$ is a $p,q$-path omitting  $x_2$. 

For $p\in W$ and $q\in X_1$,  $\overrightarrow{S}[p,c_2]+\overrightarrow{R}[c_2,a_1]+\overrightarrow{P}[a_1,y_1]+\overrightarrow{e_1}+\overrightarrow{P}[x_1,q]$ is a $p,q$-path omitting $x_2$. 

For $p,q\in X_1$,    $\overrightarrow{P}[p,b_2]+\overrightarrow{S}[b_2,c_2]+\overrightarrow{R}[c_2,a_1]+\overrightarrow{P}[a_1,y_1]+\overrightarrow{e_1}+\overrightarrow{P}[x_1,q]$ is a $p,q$-walk omitting $x_2$.   

For $p\in X_2$ and $q\in W$,    $\overrightarrow{P}[p,x_2]+\overrightarrow{e_2}+\overrightarrow{S}[y_2,q]$ is a $p,q$-path omitting $x_1$.

For $p\in W$ and $q\in X_2$,  $\overrightarrow{S}[p,a_2]+\overrightarrow{P}[a_2,q]$  is a $p,q$-path omitting $x_1$.  

For $p,q\in X_2$,  $\overrightarrow{P}[p,x_2]+\overrightarrow{e_2}+  \overrightarrow{S}[y_2,a_2]+\overrightarrow{P}[a_2,q]$ is a $p,q$-walk omitting $x_1$. 

For $p,q\in W$, $\overrightarrow{S}[p,q]$ is a $p,q$-path omitting $x_1$ and $x_2$.

For $p\in X_1$ and $q\in X_2$,   $\overrightarrow{P}[p,b_2]+\overleftarrow{R}[a_2,b_2]+\overrightarrow{P}[a_2,q]$  is a $p,q$-path omitting some vertex in  $\overrightarrow{P}(b_2,a_2)$. 

For $p\in X_2$ and $q\in X_1$,  $\overrightarrow{P}[p,x_2]+\overrightarrow{e_2}+\overleftarrow{P}[c_2,y_2]+\overrightarrow{R}[c_2,a_1]+\overrightarrow{P}[a_1,y_1]+\overrightarrow{e_1}+\overrightarrow{P}[x_1,q]$  is a $p,q$-path omitting $b_1$ since $b_2\neq c_2$. 

\textit{Case 2.2.3.} $a_2=y_2$ and $b_2=c_2$.

\noindent Let $\overrightarrow{S}=\overleftarrow{P}[a_1,b_1]\cup\overleftarrow{R}[b_1,a_1]$ be the union of the reverse paths of $\overrightarrow{P}[a_1,b_1]$ and  $\overrightarrow{R}[b_1,a_1]$. Note that $\overrightarrow{S}$ is a cycle. We orient a spanning subgraph of $G$ as follows:

\begin{center}
	$\overrightarrow{H}=\overrightarrow{P}[x_1,a_1]\cup \overrightarrow{S}\cup\overrightarrow{R}[a_2,c_1]\cup \overrightarrow{P}[b_1,x_2]\cup\{\overrightarrow{e_1},\overrightarrow{e_2}\}$,
\end{center}

\noindent and set 

\begin{center}
	$X_1=V(\overrightarrow{P}[x_1,a_1))$, $W=V(\overrightarrow{P}[a_1,b_1])$, $X_2=V(\overrightarrow{P}(b_1,x_2])$.
\end{center}

\noindent We now check that for any two vertices $p$ and $q$ in $G$, there exists a $p,q$-walk omitting some vertex.

For $p\in X_1$ and $q\in W$,  $\overrightarrow{P}[p,a_1]+\overrightarrow{S}[a_1,q]$ is a $p,q$-path omitting  $x_2$. 

For $p\in W$ and $q\in X_1$, $\overrightarrow{S}[p,y_1]+\overrightarrow{e_1}+\overrightarrow{P}[x_1,q]$ is a $p,q$-path omitting $x_2$. 

For $p,q\in X_1$, $\overrightarrow{P}[p,a_1]+\overrightarrow{S}[a_1,y_1]+\overrightarrow{e_1}+\overrightarrow{P}[x_1,q]$ is a $p,q$-walk omitting $x_2$.   

For $p\in X_2$ and $q\in W$,  $\overrightarrow{P}[p,x_2]+\overrightarrow{e_2}+\overrightarrow{R}[y_2,c_1]+\overrightarrow{S}[c_1,q]$ is a $p,q$-walk omitting $x_1$.

For $p\in W$ and $q\in X_2$, $\overrightarrow{S}[p,b_1]+\overrightarrow{P}[b_1,q]$  is a $p,q$-path omitting $x_1$.  

For $p,q\in X_2$, $\overrightarrow{P}[p,x_2]+\overrightarrow{e_2}+\overrightarrow{R}[y_2,c_1]+  \overrightarrow{S}[c_1,b_1]+\overrightarrow{P}[b_1,q]$ is a $p,q$-walk omitting $x_1$. 

For $p,q\in W$, $\overrightarrow{S}[p,q]$ is a $p,q$-path omitting $x_1$ and $x_2$.

For $p\in X_1$ and $q\in X_2$, $\overrightarrow{P}[p,a_1]+\overleftarrow{R}[b _1,a_1]+\overrightarrow{P}[b_1,q]$  is a $p,q$-path omitting some vertex in  $\overrightarrow{P}(a_1,b_1)$. 

For $p\in X_2$ and $q\in X_1$,
	\begin{itemize}
		\item if $p\in V(\overrightarrow{P}(a_2,x_2])$, then 
		\begin{center}
			 $\overrightarrow{P}[p,x_2]+\overrightarrow{e_2}+\overrightarrow{R}[a_2,c_1]+\overrightarrow{S}[c_1,y_1]+\overrightarrow{e_1}+\overrightarrow{P}[x_1,q]$  
		\end{center}
		is a $p,q$-path omitting some vertex in  $\overrightarrow{P}(b_2,a_2)$,
		\item if $p\in V(\overrightarrow{P}(b_1,a_2])$, then 
		\begin{center}
			$\overrightarrow{P}[p,a_2]+\overrightarrow{R}[a_2,c_1]+\overrightarrow{S}[c_1,y_1]+\overrightarrow{e_1}+\overrightarrow{P}[x_1,q]$  
		\end{center}
		is a $p,q$-walk omitting $x_2$.
	\end{itemize}

Hence, $G$ has an orientation of diameter at most $n-2$.
\end{proof}

Theorem~\ref{thm:n-2} follows from Lemmas~\ref{lem:Hamiltonian} and~\ref{lem:non-Hamiltonian}.

\section{Concluding Remarks}\label{section-concludingremark} 

Conjecture~\ref{con:m(n,d)} has been proved for the case $d=2$ by Cochran,  Czabarka, Dankelmann and Sz\'{e}kely  and for the case $d=n-2$ by Theorems~\ref{thm:n-2} and~\ref{thm:sharp}. The case $3\leq d\leq n-3$ remains open. Theorem~\ref{thm:sharp} shows that $m(n,d)\geq{n-d\choose 2}+n+2$ for the case $5\leq d\leq n-2$. We cannot prove the lower bound for the case $3\leq d\leq4$.

\begin{prob}
	For $n\geq5$ and $3\leq d\leq4$, construct a graph of order $n$ and size ${n-d\choose 2}+n+1$ without an orientation of diameter at most $d$.
\end{prob} 

To prove Conjecture~\ref{con:m(n,d)}, it remains to show that  $m(n,d)\leq{n-d\choose 2}+n+2$. 

\begin{prob}
	For $3\leq d\leq n-3$, prove that every bridgeless graph of order $n$ and size at least ${n-d\choose 2}+n+2$  has an orientation of diameter at most $d$.
\end{prob} 
	
\section*{Acknowledgment} The study was supported by Thammasat Univerity Research Fund, Contact No TUFT 54/2566.

	\bibliographystyle{siam} 
	\bibliography{EDO}
	
\end{document}